\documentclass{article}
\usepackage[utf8]{inputenc}
\usepackage{amssymb, amsmath, amsthm}

\newtheorem{theorem}{Theorem}[section]

\newtheorem{lemma}[theorem]{Lemma}
\newtheorem{definition}[theorem]{Definition}
\newtheorem{proposition}[theorem]{Proposition}

\newtheorem{remark}{Remark}

\title{On the number of dot product chains in finite fields and rings}
\author{Vincent Blevins, David Crosby, Ethan Lynch, and Steven Senger}
\date{\today}

\begin{document}

\maketitle
\begin{abstract}
We explore variants of Erd\H os' unit distance problem concerning dot products between successive pairs of points chosen from a large finite subset of either $\mathbb F_q^d$ or $\mathbb Z_q^d,$ where $q$ is a power of an odd prime. Specifically, given a large finite set of points $E$, and a sequence of elements of the base field (or ring) $(\alpha_1,\ldots,\alpha_k)$, we give conditions guaranteeing the expected number of $(k+1)$-tuples of distinct points $(x_1,\dots, x_{k+1})\in E^{k+1}$ satisfying $x_j \cdot x_{j+1}=\alpha_j$ for every $1\leq j \leq k$.
\end{abstract}
\section{Introduction}
\subsection{Background}
In 1946, in \cite{Erd46}, Paul Erd\H os raised a question that eludes mathematicians to this day: In a large finite point set in the plane, how many pairs of points can be separated by the same distance? This is often referred to as the unit distance problem. While optimal bounds are not yet known, there has been much activity on this and related problems. See, for example, \cite{BMP, GK, SST}.

A more general family of questions is: Given a large subset of an ambient set with some structure (such as a vector space or module), how many instances of some class of point configurations can be present? Let $q$ be some power of an odd prime, $p,$ and consider $d$-dimensional vector spaces over finite fields, $\mathbb F_q^d,$ or the $d$-rank free modules over $\mathbb Z_q^d,$ instead of the plane. There has been much work studying different types of analogs of the unit distance problem in these settings, as can be seen in \cite{IR07}, by Alex Iosevich and Misha Rudnev.

While Erd\H os asked about pairs of points determining a fixed distance, one can investigate $k$-tuples of points determining other functions, such as dot products. Again, there is an abundance of work on such generalizations. See \cite{CEHIK, CHISU, HIKR, IS, Vinh} and the references contained therein.
\newpage
\subsection{Notation}
We now give a precise definition of our main object of study. We follow the convention of similar definitions from those given in \cite{BIT, PSS}.

\begin{definition}
Given a sequence of field (or ring) elements $(\alpha_1,\ldots,\alpha_k)$, a {\bf dot product $k$-chain} is a $(k+1)$-tuple of distinct points $(x_1,\dots, x_{k+1})$ satisfying $x_j \cdot x_{j+1}=\alpha_j$ for every $1\leq j \leq k$.
\end{definition}

In \cite{BS}, Daniel Barker and the fourth listed author gave bounds on the number of dot product 2-chains in the plane. These have been expanded and generalized recently in \cite{KMS} and \cite{GPRS} This approach was adapted to the settings of $\mathbb F_q^d$ and $\mathbb Z_q^d$ by Dave Covert and the fourth listed author in \cite{CS}.

In this note, we study $k$-chains in these settings. Given a subset of $E\subseteq \mathbb F_q^d$ (or $\mathbb Z_q^d$), and a $k$-tuple of elements $\alpha = (\alpha_1, \dots, \alpha_k)\in \mathbb F_q^k$ (or $\mathbb Z_q^k$), let $\Pi_\alpha(E)$ denote the set of dot product $k$-chains corresponding to $\alpha$ whose members are points in $E$. Throughout this paper, we assume that $k$ is like a constant compared to the size of any subset $E$. To ease exposition, we use the asymptotic symbols $X\lesssim Y$ if $X=O(Y),$ and $X\approx Y$ when $X = \Theta(Y).$ Moreover, we write $X \gtrapprox Y$ when for every $\epsilon >0,$ there exists a constant $C_\epsilon > 0$ such that $X \gtrsim C_\epsilon q^\epsilon Y.$

\subsection{Main results}
Our first new result is on 3-chains in $\mathbb F_q^d.$

\begin{theorem}\label{3chains}
Given $E\subseteq \mathbb F_q^d,$ with $q$ a power of an odd prime, and $\alpha \in \mathbb F_q^3,$ if $|E|\gtrapprox q^\frac{d+3}{2}$ (or $|E|\gtrapprox q^\frac{d+2}{2}$ if one of the $\alpha_j$ is nonzero, or $|E|\gtrapprox q^\frac{d+1}{2}$ if the $\alpha_j$ are nonzero), then
$$\left|\Pi_\alpha(E)\right|= (1+o(1)) \frac{|E|^4}{q^3}.$$
\end{theorem}

The method of proof is to count the number of dot product 3-chains with a character sum, which we break up into a main term and several remainder terms. The main idea is to estimate the remainder terms. Using this same line of reasoning, but with less strict hypotheses, we get the following results which are more general, but weaker.

\begin{theorem}\label{kChainsR}
Given $E\subseteq \mathbb Z_q^d,$ with $q=p^\ell$ a power of an odd prime $p$, $k\in \mathbb N$, and a $k$-tuple of units, $\alpha \in \mathbb Z_q^k,$ if $|E|\gtrapprox q^{\frac{d(2\ell - 1) + 1}{2\ell} + \frac{k-2}{2}}$, then
$$\left|\Pi_\alpha(E)\right|=(1 + o(1)) \frac{|E|^{k+1}}{q^k}.$$
\end{theorem}

\begin{theorem}\label{kChainsF}
Given $E\subseteq \mathbb F_q^d,$ with $q$ a power of an odd prime, and $\alpha \in \mathbb F_q^k,$ if \(|E|\gtrapprox q^{\frac{d+k}{2}}\) (or \(|E| \gtrapprox q^{\frac{d+k-1}{2}}\) if the components of \( \mathbf{\alpha}\) are nonzero), then
$$\left|\Pi_\alpha(E)\right|= (1+o(1))\frac{|E|^{k+1}}{q^k}.$$
\end{theorem}

While all of the results so far have stated that particular types of dot product $k$-chains occur with regularity within large enough subsets, we now turn to upper bounds on the number of a given type of dot product $k$-chain in a subset of $\mathbb F_q^2$ in terms of the size of the subset.

\begin{theorem}\label{smallSet}
Given $E\subseteq \mathbb F_q^2,$ with $q$ a power of an odd prime, and $\alpha \in \mathbb F_q^k,$ with all components nonzero,
$$\left|\Pi_\alpha(E)\right|\lesssim |E|^{\left\lceil \frac{2(k+1)}{3}\right\rceil}.$$
\end{theorem}

These bounds hold without strict size constraints on the size of $E$ found in the estimates above. However, if the size of $E$ is sufficiently large to apply Theorem \ref{3chains} or \ref{kChainsF}, then Theorem \ref{smallSet} will yield weaker bounds. This result is a straightforward corollary of an estimate in \cite{CS} (Theorem 3 in that paper), though that result has a mistake as stated. The finite fields result is true, but there is a fatal flaw in the way that incidences of points and lines were counted. This is detailed in Section \ref{smallSection}.

\subsection{Sharpness and relevance of hypotheses}

While these results are almost certainly far from sharp, we provide some constructions that demonstrate some of what is known. These also help indicate why size conditions on $E$ make sense as hypotheses, as comparatively small sets can exhibit behavior far from what would be expected from a random subset. We also show how the zero dot product has distinct behavior, to demonstrate why we have the relevant hypotheses on the $\alpha_j$ being units or at least nonzero.

\begin{remark}
Consider the set $$E:=(\{(x,0):x\in \mathbb F_q\})\cup(\{(0,y):y\in \mathbb F_q\})\subseteq \mathbb F_q^2.$$
Clearly this set has $\gtrsim q^{k+1}$ dot product $k$-chains of dot product zero, obtained by alternately selecting points from each of the subsets listed in the union. Here, $|E|\approx q.$
\end{remark}

While the previous remark shows that zero dot products can exhibit markedly different behavior, we also have related examples for nonzero dot products.

\begin{remark}
Consider the set $$E:=(\{(x,0,\alpha):x\in \mathbb F_q\})\cup(\{(0,y, 1):y\in \mathbb F_q\})\subseteq \mathbb F_q^3.$$
Clearly this set has $\gtrsim q^{k+1}$ dot product $k$-chains of dot product $\alpha,$ obtained by alternately selecting points from each of the subsets listed in the union. Again, we have $|E|\approx q.$
\end{remark}

Clearly, both of these examples can be modified to fit into $\mathbb F_q^d$ or $\mathbb Z_q^d$ for any $d$ large enough to admit an embedding. This indicates that the occurrence of more dot product $k$ chains of a given type is an artifact of ``lower dimensional" subsets. If $E$ is a large enough subset of the ambient space (or module), then these occurrences are outweighed by the behavior of the rest of the set $E$.

\section{Bounds on 3-chains in $\mathbb F_q^d$}

Let \(\chi(\alpha)\) to denote the canonical additive character of \(\mathbb F_q\). The plan will be to count the number of 3-chains using a character sum. We will then split this sum into a main term and several error terms. The key idea will be getting nontrivial bounds on the error terms based on the size of the set $E$.

For 3-chains in \(\mathbb{F}_q^d\):

\[\left|\Pi_{\alpha_1,\alpha_2,\alpha_3}{(E)}\right| = |\{(x_1,x_2,x_3,x_4) \in E^4 : x_j\cdot x_{j+1}=\alpha_j\}|\]
\[=q^{-3}\sum_{s_j}\sum_{x_j}\chi(s_j(x_j\cdot x_{j+1}-\alpha_j))\]
\[=M+R_{1}+R_{2}+R_{3}.\]
Here, \(M\) is the case where all auxiliary variables are zero, and each \(R_n\) is the case where \(n\) auxiliary variables are non-zero.  Moreover \(R_{\{n_i,\dots,n_j\}}\) is the case where the set of specific auxiliary variables, \(n_i, \dots, n_j\) are nonzero.
\[M=q^{-3}\sum_{s_j=0}\sum_{x_j\in E}\chi(s_j(x_j\cdot x_{j+1}-\alpha_j)) \]
\[=q^{-3}\sum_{x_j\in E}\chi(0\cdot(x_j\cdot x_{j+1}-\alpha_j)) \]
\[=q^{-3}\sum_{x_j\in E}\chi(0)\]
\[=q^{-3}|E|^4.\]
The hope is that we have \(M\) be the dominant term in the sum so that it can serve as our estimate for \(|\Pi_{\alpha_1, \alpha_2, \alpha_3}(E)|,\) while the other terms can be bounded.
\[R_1 = q^{-3}\sum_{\substack{s_2=s_3=0\\ s_1\in \mathbb F_q^*}}\sum_{x_j\in E}\chi(s_j(x_j\cdot x_{j+1}-\alpha_j)) +q^{-3}\sum_{\substack{s_1=s_3=0\\ s_2\in \mathbb F_q^*}}\sum_{x_j\in E}\chi(s_j(x_j\cdot x_{j+1}-\alpha_j))+ \]
\[q^{-3}\sum_{\substack{s_1=s_2=0\\ s_3\in \mathbb F_q^*}}\sum_{x_j\in E}\chi(s_j(x_j\cdot x_{j+1}-\alpha_j))\]
\[= R_{\{1\}}+R_{\{2\}}+R_{\{3\}}.\]
\\
We now introduce the following result due to Derrick Hart, Alex Iosevich, Doo Won Koh, and Misha Rudnev, in \cite{HIKR}, to help us get a handle on quantities like $R_{\{1\}}.$
\begin{lemma}\label{1dp}[Equation 2.5 from \cite{HIKR}] For any set $E \subseteq \mathbb{F}_q^d$, we have the bound
\begin{equation}\label{ell1}
\left| \sum_{s \neq 0} \sum_{x,y \in E} \chi(s(x \cdot y - \gamma))\right| \leq  |E|q^{\frac{d+1}{2}}\lambda(\gamma),
\end{equation}
where $\lambda(\gamma) = 1$ for $\gamma \in \mathbb{F}_q^*$ and $\lambda(0) = \sqrt{q}$.  
\end{lemma} 
\[R_{\{1\}}=q^{-3}\sum_{\substack{s_2=s_3=0\\ s_1\in \mathbb F_q^*}}\sum_{x_j\in E}\chi(s_j(x_j\cdot x_{j+1}-\alpha_j)).\]
\[=q^{-3}|E|^2\sum_{s_1\in \mathbb{F}^*_q}\sum_{x_1,x_2\in E}\chi(s_j(x_j\cdot x_{j+1}-\alpha_j)). \]
by Lemma \ref{1dp} we have
\[|R_{\{1\}}|\leq q^{-3}|E|^2\cdot|E|q^{\frac{d+1}{2}}\lambda(\alpha_1)\]
\[=|E|^3q^{\frac{d-5}{2}}\lambda(\alpha_1). \]

By similar arguments, we also get that
\[\left|R_{\{2\}}\right|\leq |E|^3q^{\frac{d-5}{2}}\lambda(\alpha_2), \text{ and }\]
\[\left|R_{\{3\}}\right|\leq|E|^3q^{\frac{d-5}{2}}\lambda(\alpha_3).\]

So, 

\[\left|R_1\right| = \left|R_{\{1\}} + R_{\{2\}} + R_{\{3\}}\right| \]
\[\leq |E|^3q^{\frac{d-5}{2}}(\lambda(\alpha_1)+\lambda(\alpha_2)+\lambda(\alpha_3)).  \]

Now let's look at the case with two nonzero auxiliary variables. This case in particular is special because we have two sub-cases. The first is when the two dot products share a point, and the second is when they share no points in common.

\[R_2 = R_{\{1,2\}} + R_{\{1,3\}} + R_{\{2,3\}}\]

We'll start with \(R_{\{1,2\}}\) and \(R_{\{2,3\}}\) which are similar in that the auxiliary variables involved are consecutive. 

\[R_{\{1,2\}} = q^{-3}|E|\sum_{s_1,s_2\in \mathbb{F}^*_q}\sum_{x_1,x_2,x_3\in E}\chi(s_1(x_1\cdot x_2-\alpha_1))\chi(s_2(x_2\cdot x_3-\alpha_2))\]
\[=q^{-3}|E|\sum_{s_1,s_2\in \mathbb{F}^*_q}\sum_{x_1,x_2,x_3\in E}\chi(s_1(x_1\cdot x_2-\alpha_1))\chi(s_2(x_2\cdot x_3-\alpha_2))\]
\[=q^{-3}|E|\cdot T_1(E),\]
where we define
\[
T_i(E):=\sum_{s_i, s_{i+1} \in \mathbb F_q^*} \left(\sum_{\substack{x_i, x_{i+1},\\ x_{i+2} \in E}} \chi(s_i(x_i \cdot x_{i+1} - \alpha_i) )   \chi(s_{i+1}(x_{i+1} \cdot x_{i+2} - \alpha_{i+1}))\right)
\]
We state an adaptation of one of the key estimates in \cite{CS}, due to Dave Covert and the fourth listed author as a lemma.

\begin{lemma}[Estimate of $III$ from \cite{CS}]\label{2dp}
\[\left|T_i(E)\right| \lesssim q^{d+1}|E|\lambda(\alpha_i)\lambda(\alpha_{i+1}).\]
\end{lemma}

By the above argument and Lemma \ref{2dp}, we get
\[\left|R_{\{1,2\}}\right| \leq q^{-3}|E|\left| T_1(E)\right| \leq q^{-3}|E|\cdot q^{d+1}|E|\lambda(\alpha_1)\lambda(\alpha_2),\]
which yields
\begin{equation}\label{R12}
\left|R_{\{1,2\}}\right| \leq q^{d-2}|E|^2\lambda(\alpha_1)\lambda(\alpha_2).
\end{equation}
Similarly, we can see that
\[\left|R_{\{2,3\}}\right| \leq q^{d-2}|E|^2\lambda(\alpha_2)\lambda(\alpha_3).\]

Next we'll look at \(R_{\{1,3\}}\).
\[R_{\{1,3\}} = q^{-3}\sum_{s_1,s_3\in \mathbb{F}^*_q}\sum_{x_j\in E}\chi(s_1(x_1\cdot x_2-\alpha_1))\chi(s_3(x_3\cdot x_4-\alpha_3)).\]
By taking absolute values and appealing to Cauchy-Schwarz, we get
\[\left|R_{\{1,3\}} \right| \leq q^{-3}\left|\sum_{x_j\in E}\sum_{s_1\in \mathbb{F}^*_q}\chi(2s_1(x_1\cdot x_2-\alpha_1))\right|^\frac{1}{2}\left|\sum_{x_j\in E}\sum_{s_3\in \mathbb{F}^*_q}\chi(2s_3(x_3\cdot x_4-\alpha_3))\right|^{\frac{1}{2}},\]
where we write $2s_j$ to mean $s_j+s_j.$ Since $s_1$ and $s_3$ are ranging over \(\mathbb{F}_q^*\), and $q$ is odd, we know that \(2s_j\neq 0\). So the the sums over $s_1$ and $s_3$ will just be over a permutation of the elements of $\mathbb F_q^*.$ Therefore, by a change of variables, $t_j=2s_j,$ the above quantity can be written
\[= q^{-3}\left|\sum_{x_j\in E}\sum_{t_1\in \mathbb{F}^*_q}\chi(t_1(x_1\cdot x_2-\alpha_1))\right|^\frac{1}{2}\left|\sum_{x_j\in E}\sum_{t_3\in \mathbb{F}^*_q}\chi(t_3(x_3\cdot x_4-\alpha_3))\right|^{\frac{1}{2}}.\]
Since $x_3$ and $x_4$ are left out of the first sum, and $x_1$ and $x_2$ are left out of the second, we can write this as 
\[= q^{-3}\left||E|^2\sum_{x_1, x_2\in E}\sum_{t_1\in \mathbb{F}^*_q}\chi(t_1(x_1\cdot x_2-\alpha_1))\right|^\frac{1}{2}\left||E|^2\sum_{x_3, x_4\in E}\sum_{t_3\in \mathbb{F}^*_q}\chi(t_3(x_3\cdot x_4-\alpha_3))\right|^{\frac{1}{2}}\]
\[= q^{-3}|E|^2\left|\sum_{x_1, x_2\in E}\sum_{t_1\in \mathbb{F}^*_q}\chi(t_1(x_1\cdot x_2-\alpha_1))\right|^\frac{1}{2}\left|\sum_{x_3, x_4\in E}\sum_{t_3\in \mathbb{F}^*_q}\chi(t_3(x_3\cdot x_4-\alpha_3))\right|^{\frac{1}{2}}.\]
Applying Lemma \ref{1dp} twice as in our bound of \(R_1\), we get that this is bounded above by
\[\leq q^{-3}|E|^3\left(|E|q^\frac{d+1}{2}\lambda(\alpha_1)\right)^\frac{1}{2}\left(|E|q^\frac{d+1}{2}\lambda(\alpha_3)\right)^{\frac{1}{2}}\]
\[\leq q^\frac{d-5}{2}|E|^3\sqrt{\lambda(\alpha_1)\lambda(\alpha_3)}.\]
Combining the above bounds on the components of $R_2$ yields
\[
\left|R_2\right|\leq q^\frac{d-5}{2}|E|^3\sqrt{\lambda(\alpha_1)\lambda(\alpha_2)+\lambda(\alpha_1)\lambda(\alpha_3)+\lambda(\alpha_2)\lambda(\alpha_3)}.
\]

Now we can start bounding \(R_{3}\). The strategy here will be to add in some terms with auxiliary variables equal to zero, and break the new sum up into pieces that we can estimate separately. We recall 
\(R_3=R_{\{1,2,3\}},\) so we write
\[R_3 = q^{-3}\sum_{s_1,s_2,s_3\in \mathbb F^*_q}\sum_{x_j\in E}\chi(s_1(x_1\cdot x_2-\alpha_1))\chi(s_2(x_2\cdot x_3-\alpha_2))\chi(s_3(x_3\cdot x_4-\alpha_3))\]
We define the secondary portion \(R_3'\) is made up of terms similar to those of $R_3,$ but with \(s_3=0\). We follow this process unless $\alpha_3$ is the only nonzero $\alpha_j,$ in which case we follow the same procedure as below, but reversing the roles of $\alpha_1$ and $\alpha_3$, and the corresponding variables. This allows us to have fewer lambda factors contributing $\sqrt q$ to our final bound.
\[R_3':= q^{-3}\sum_{s_1,s_2 \in \mathbb F_q^*}\sum_{x_j\in E}\chi(s_1(x_1\cdot x_2-\alpha_1))\chi(s_2(x_2\cdot x_3-\alpha_2))\chi(0(x_3\cdot x_4-\alpha_3)).\]
We can apply the triangle inequality and recall that $\chi(0)=1$ to get
\[
\left|R_3'\right|= q^{-3}|E|\left|\sum_{s_1,s_2 \in \mathbb F_q^*}\sum_{x_j\in E}\chi(s_1(x_1\cdot x_2-\alpha_1))\chi(s_2(x_2\cdot x_3-\alpha_2))\right|.
\]
Notice that the inner sum is exactly $T_1(E)$, so we can appeal to Lemma \ref{2dp} to get that
\begin{equation}\label{R'}
\left|R_3'\right|\leq q^{d-2}|E|^2\lambda(\alpha_1)\lambda(\alpha_2).
\end{equation}
By the triangle inequality, we get that
\[|R_3|=|R_3+R_3'-R_3'| \leq |R_3+R_3'|+|R_3'|\]

The first term in this sum involves \(R_3+R_3'\), which we write as
\[=q^{-3}\sum_{s_1,s_2 \in \mathbb F_q^*}\sum_{x_j\in E}\chi(s_1(x_1\cdot x_2-\alpha_1))\chi(s_2(x_2\cdot x_3-\alpha_2))\sum_{s_3\in \mathbb F_q}\chi(s_3(x_3\cdot x_4-\alpha_3)).\]
Here, we can re-associate and get the sum in two parts, the sum of terms where \(x_3\cdot x_4 = \alpha_3\) and the terms where \(x_3\cdot x_4 \neq \alpha_3\). In the case where \(x_3\cdot x_4 = \alpha_3\), the inner sum collapses to \(q\), as we get \(\chi(0)=1\) for each element of \(\mathbb F_q\). In the case where \(x_3\cdot x_4 \neq \alpha_3\), we get zero by orthogonality. Altogether, we get 

\[R_3+R_3'=q\cdot q^{-3}\sum_{s_1,s_2 \in \mathbb F_q^*}\sum_{x_j\in E}\chi(s_1(x_1\cdot x_2-\alpha_1))\chi(s_2(x_2\cdot x_3-\alpha_2)) +\]
\[ 0\cdot q^{-3}\sum_{s_1,s_2 \in \mathbb F_q^*}\sum_{x_j\in E}\chi(s_1(x_1\cdot x_2-\alpha_1))\chi(s_2(x_2\cdot x_3-\alpha_2)).\]
We can ignore the second sum as it is multiplied by zero. Now we take absolute values on both sides to get
\[\left| R_3+R_3'\right|=q\left|q^{-3}\sum_{s_1,s_2 \in \mathbb F_q^*}\sum_{x_j\in E}\chi(s_1(x_1\cdot x_2-\alpha_1))\chi(s_2(x_2\cdot x_3-\alpha_2))\right|\]
\[=q\cdot \left|R_{\{1,2\}}\right| \leq q^{d-1}|E|^2\lambda(\alpha_1)\lambda(\alpha_2),\]
where we applied \eqref{R12} in the last step. Comparing the estimates of \(|R_1|, |R_2|,\) and \(|R_3|\) to \(M\) yields the desired result.

\section{Bounds on \(k\)-Chains in $\mathbb Z_q^d$}

In this proof, we provide an asymptotic bound for the number of dot product \(k\)-chains of units in a sufficiently large finite subset of
\(\mathbb{Z}_q^d\) (the \(d\) rank free module over \(\mathbb{Z}_q\)) where \(d\) is a positive integer and \(q = p^\ell\) for some odd prime \(p\) and positive integer \(\ell\). Given a subset \(E\) of \(\mathbb{Z}_q^d\) and a \(k\)-tuple of units \(\mathbf{\alpha} = (\alpha_1, \cdots, \alpha_k)\) in \(\mathbb{Z}_q ^\times\), denote the set of \(k\)-chains in \(E\) by
\(\Pi_\alpha (E)\). The asymptotic bound for \(\Pi_\alpha (E)\) is as follows:

\begin{theorem}\label{mainR}
	Let \(E \subseteq \mathbb{Z}_q ^d\) where \(q=p^\ell\) a power of an odd prime $p$, and let \(\mathbf{\alpha} = (\alpha_1, \alpha_2, \cdots, \alpha_k)\) be a \(k\)-tuple of units in \(\mathbb{Z}_q\) where \(k \geq 2\). Then we have
\[
|\Pi_\alpha (E)| = \frac{|E|^{k+1}}{q^k}(1 + o(1))
\]
provided \(|E| \gtrapprox q^{\frac{d(2\ell - 1) + 1}{2\ell} + \frac{k-2}{2}}\)
\end{theorem}


The proof of this bound relies on previous results which we restate here for the reader's convenience. We first define a useful quantity. For $E \subseteq \mathbb Z_q^d$, and $\alpha$ a unit in $\mathbb Z_q,$ let
\[
S_{E, \alpha}(x) := \sum_{s \in \mathbb{Z}_q \setminus \{0\}} \sum_{y \in E} \chi(s(x \cdot y-\alpha))
\]
The following estimate is due to Dave Covert, Alex Iosevich, and Jonathan Pakianathan, in \cite{CIP}.
\begin{lemma}\label{1dpR}[from \cite{CIP}]
Suppose that \(E \subseteq \mathbb{Z}_q^d\), where \(q = p^\ell\) is a power of an odd prime. Let
\(\gamma \in \mathbb{Z}_q^\times\) be a unit. Then we have the following upper bound:
\[
\sum_{x \in E}S_{E,\alpha}(x) \leq 2|E|q^{(\frac{d-1}{2})(2 - \frac{1}{\ell}) + 1}.
\]
\end{lemma}
Next, notice that by the definition of $S_{E, \alpha}(x)$ and Cauchy-Schwarz, we get
\begin{equation}\label{s1}
\left|\sum_{x\in \mathbb Z_q^d}S_{E, \alpha}(x)S_{E,\beta}(x)\right| \leq \left( \sum_{x \in \mathbb{Z}_q^d} \left\lvert S_{E,\alpha}(x) \right\rvert ^2\right)^{\frac{1}{2}}\cdot
\left( \sum_{x \in \mathbb{Z}_q^d} \left\lvert S_{E, \beta}(x) \right\rvert ^2\right)^{\frac{1}{2}}
\end{equation}
We now record a technical estimate from \cite{CS} (the bound of $III$ in $\mathbb Z_q^d$), which we state here as a lemma. 
\begin{lemma}\label{2dpR}[from \cite{CS}]
Suppose that \(E \subseteq \mathbb{Z}_q^d\), where \(q = p^\ell\) is a power of an odd prime. Let \(\alpha, \beta \in \mathbb{Z}_q^\times\) be units. We have the following upper bounds:

\begin{equation}\label{S2}
\left( \sum_{x \in \mathbb{Z}_q^d} \left\lvert S_{E,\alpha}(x) \right\rvert ^2\right)^{\frac{1}{2}}
\left( \sum_{x \in \mathbb{Z}_q^d} \left\lvert S_{E, \beta}(x) \right\rvert ^2\right)^{\frac{1}{2}}\leq 2|E|q^{\frac{d(2\ell - 1)}{\ell} + \frac{1}{\ell}}.
\end{equation}
\end{lemma}

\begin{proof}[Proof of Main Result]
Let \(q = p^\ell\) be a power of an odd prime, and consider a subset \(E\) of \(\mathbb{Z}_q^d\) where \(d\) is a positive integer. If \(\alpha = (\alpha_1, \cdots, \alpha_k)\) is a $k$-tuple of units (for an integer \(k > 2\)), then we have by orthogonality
\[
|\Pi_\alpha (E)| = q^{-k} \sum_{\substack{s_j \in \mathbb{Z}_q \\ 1 \leq j \leq k}}\sum_{\substack{x_j \in E \\ 1 \leq j \leq k+1}}\prod_{j=1}^{k} \chi(s_j(x_j \cdot x_{j+1} - \alpha_j)).
\]
As above, we proceed by decomposing this sum into others based on which \(s_i\) are zero. Given a binary $k$-tuple \textbf{j}, let \(R_{\mathbf{j}}\) be the component sum where \(s_i = 0\) if \textbf{j}(i) \(= 0\) and \(s_i \neq 0\) if \textbf{j}(i) \(= 1\). Denote the sum of the \(R_{\mathbf{j}}\) where exactly $n$ entries of \textbf{j} are nonzero by \(R_n\). First, consider the sum \(R_0\) where every \(s_i = 0\). Because $\chi(0)=1,$ we get
\begin{align*}
R_0 &= q^{-k} \sum_{\substack{s_j = 0 \\ 1 \leq j \leq k}} \sum_{\substack{x_j \in E \\ 1 \leq j \leq k+1}}\prod_{j=1}^{k} \chi(s_j(x_j \cdot x_{j+1} - \alpha_j)) \\
&= q^{-k}|E|^{k+1}.			
\end{align*}
Now let \textbf{j} be a binary $k$-tuple with the $i$th entry equal to 1 and the others zero. After simplifying, it follows from Lemma \ref{1dpR} that
\begin{align*}
|R_{\mathbf{j}}| &= q^{-k} \sum_{\substack{s_j = 0 \\ 1 \leq j \leq k\\ j\neq i}} \sum_{\substack{x_j \in E \\ 1 \leq j \leq k+1}}\prod_{j=1}^{k} \chi(s_j(x_j \cdot x_{j+1} - \alpha_j))\sum_{s_i\neq 0}S_{E,\alpha_i}(x_i)\\
&\leq q^{-k} \sum_{\substack{s_j = 0 \\ 1 \leq j \leq k\\ j\neq i}} \sum_{\substack{x_j \in E \\ 1 \leq j \leq k+1}}\prod_{j=1}^{k} \chi(s_j(x_j \cdot x_{j+1} - \alpha_j)) \left(2|E|q^{(\frac{d-1}{2})(2 - \frac{1}{\ell}) + 1}\right)  \\
&= 2q^{-k}|E|^{k}q^{(\frac{d-1}{2})(2 - \frac{1}{\ell}) + 1}.
\end{align*}
As there are \( \binom{k}{2} = k\) such tuples, it follows from the triangle inequality that
\begin{equation}\label{R1}
|R_1| \leq 2kq^{-k}|E|^{k}q^{(\frac{d-1}{2})(2 - \frac{1}{\ell}) + 1}.
\end{equation}
Next, consider a sum \(R_{\mathbf{j}}\) where exactly two entries of \textbf{j} are nonzero. If the nonzero entries of \textbf{j} are consecutive, say entries numbered $i$ and $i+1$, then, reasoning as before, we can apply the triangle inequality, \eqref{s1}, and Lemma \ref{2dpR}, and so
\begin{align*}
|R_{\mathbf{j}}| &\leq q^{-k} \sum_{\substack{s_j = 0 \\ 1 \leq j \leq k\\ j\neq i,i+1}} \sum_{\substack{x_j \in E \\ 1 \leq j \leq k+1}}\prod_{j=1}^{k} \chi(s_j(x_j \cdot x_{j+1} - \alpha_j))\sum_{s_i,s_{i+1}\neq 0}S_{E,\alpha_i}(x_i)S_{E,\alpha_i}(x_i)\\
&\leq q^{-k} \sum_{\substack{s_j = 0 \\ 1 \leq j \leq k\\ j\neq i,i+1}} \sum_{\substack{x_j \in E \\ 1 \leq j \leq k+1}}\prod_{j=1}^{k} \chi(s_j(x_j \cdot x_{j+1} - \alpha_j))\left|\sum_{s_i,s_{i+1}\neq 0}S_{E,\alpha_i}(x_i)S_{E,\alpha_i}(x_i)\right|\\
&\leq q^{-k} \sum_{\substack{s_j = 0 \\ 1 \leq j \leq k\\ j\neq i,i+1}} \sum_{\substack{x_j \in E \\ 1 \leq j \leq k+1}}\prod_{j=1}^{k} \chi(s_j(x_j \cdot x_{j+1} - \alpha_j))\times\\
&\hskip 5em \left|\left( \sum_{x \in \mathbb{Z}_q^d} \left\lvert S_{E,\alpha}(x) \right\rvert ^2\right)^{\frac{1}{2}}
\left( \sum_{x \in \mathbb{Z}_q^d} \left\lvert S_{E, \beta}(x) \right\rvert ^2\right)^{\frac{1}{2}}\right|\\
&\leq q^{-k} \sum_{\substack{s_j = 0 \\ 1 \leq j \leq k\\ j\neq i,i+1}} \sum_{\substack{x_j \in E \\ 1 \leq j \leq k+1}}\prod_{j=1}^{k} \chi(s_j(x_j \cdot x_{j+1} - \alpha_j))\cdot\left(2|E|q^{\frac{d(2\ell - 1)}{\ell} + \frac{1}{\ell}}\right)  \\
&\leq q^{-k} |E|^{k-2}\left(2|E|q^{\frac{d(2\ell - 1)}{\ell} + \frac{1}{\ell}}\right)  \\
    	                 &= 2q^{-k}|E|^{k-1}q^{\frac{d(2\ell - 1)}{\ell} + \frac{1}{\ell}}.
\end{align*}
Putting this all together yields (in the case of consecutive nonzero entries)
\begin{equation}\label{cosecutive}
|R_{\mathbf{j}}|\leq 2q^{\frac{d(2\ell - 1)}{\ell} + \frac{1}{\ell}-k}|E|^{k-1}.
\end{equation}
However, if the nonzero entries are not consecutive, say \(\mathbf{j}(u)\), \(\mathbf{j}(v) \neq 0\) where \(1 \leq u, v \leq k\) and \(u \neq v+1,  v-1\), then we need to bound the exponent of \(|E|\) appearing in \(R_{\mathbf{j}}\) after simplifying. Let \(1 < i \leq k\) be an integer and take \(s_0 = 0\) by convention. Then if \(s_i\) and \(s_{i-1}\) are both zero, it follows for fixed \(x_{i-1}, x_{i+1} \in E\) that
\begin{align}\nonumber
&\sum_{s_i = 0} \sum_{x_i \in E} \chi(s_{i-1} (x_{i-1} \cdot x_{i} - \alpha_{i-1}))
\chi(s_i (x_i \cdot x_{i+1} - \alpha_i))\\
&\hskip 5ex= \sum_{s_i = 0} \sum_{x_i \in E}\chi(s_i (x_i \cdot x_{i+1} - \alpha_i)) = |E|,\label{consecutiveZeros}
\end{align}
and furthermore the preceding sum can be factored out of \(R_{\mathbf{j}}\). Hence, the contribution to \(R_{\mathbf{j}}\) due to these auxiliary variables is $|E|$ raised to the power of the number of integers \(1 \leq i \leq k\) such that \(s_i\) and \(s_{i-1}\) are zero.

Since the nonzero entries of \(\mathbf{j}\) are not consecutive in this case, there will be exactly four distinct \(x_i\) appearing in \(R_{\mathbf{j}}\), namely \(x_u, x_{u+1}, x_v,\) and \(x_{v+1}.\) Hence, by reasoning as in \eqref{consecutiveZeros}, the contribution to \(R_\mathbf j\) by the other \(k-3\) variables \(x_i\) and the similarly indexed auxiliary variables \(s_i\) will be \(|E|^{k-3}\), giving
\[
R_{\mathbf{j}}= q^{-k} |E|^{k-3} \left(\sum_{x_{u+1}\in E}S_{E,\alpha_u}(x_u)\right)\left(\sum_{x_{v}\in E}S_{E,\alpha_v}(x_v)\right).\]
By applying the triangle inequality, applying Cauchy-Schwarz, and dominating the sum over \(E\) by the sum over \(\mathbb Z_q^d\), we get
\[
|R_{\mathbf{j}}|\leq q^{-k} |E|^{k-3} \sum_{x_{u+1}\in E}\left|S_{E,\alpha_u}(x_u)\right|\sum_{x_{v}\in E}\left|S_{E,\alpha_v}(x_v)\right|
\]
\[
|R_{\mathbf{j}}|\leq q^{-k} |E|^{k-3} \left(\sum_{x_{u+1}\in E}\left|S_{E,\alpha_u}(x_u)\right|^2\right)^\frac{1}{2}\left(\sum_{x_{v}\in E}\left|S_{E,\alpha_v}(x_v)\right|^2\right)^\frac{1}{2}
\]
\[
|R_{\mathbf{j}}|\leq q^{-k} |E|^{k-3} \left(\sum_{x_{u+1}\in \mathbb Z_q^d}\left|S_{E,\alpha_u}(x_u)\right|^2\right)^\frac{1}{2}\left(\sum_{x_{v}\in \mathbb Z_q^d}\left|S_{E,\alpha_v}(x_v)\right|^2\right)^\frac{1}{2}
\]
We apply Lemma \ref{2dpR} to the product of sums in the last inequality to get,
\begin{equation}\label{nonconsecutiveZeros}
|R_{\mathbf{j}}| \leq 2q^{-k}|E|^{k-2}|E|q^{\frac{d(2\ell - 1)}{\ell} + \frac{1}{\ell}}= 2q^{-k}|E|^{k-1} q^{\frac{d(2\ell - 1)}{\ell} + \frac{1}{\ell}}.
\end{equation}
Combining \eqref{consecutiveZeros} and \eqref{nonconsecutiveZeros} gives us
\begin{equation}\label{R2}
|R_2| \leq 2 \binom{k}{2} q^{-k}|E|^{k-1} q^{\frac{d(2\ell - 1)}{\ell} + \frac{1}{\ell}}.
\end{equation}

Lastly, consider \(R_{\mathbf{j}}\) where \(n \geq 3\) entries of \textbf{j} are nonzero. To this end, we record a known result often used to bound quadratic forms. See \cite{Steele}.
\begin{lemma}\label{RC}
Let \(m\) and \(n\) be positive integers. Then for each double sequence 
\(\{c_{jk} : 1 \leq j \leq m, 1 \leq k \leq n\}\) and pair of sequences 
\(\{z_j: 1 \leq j \leq m\}\) \(\{y_k: 1 \leq k \leq n\}\) of complex numbers, we have the bound
\[
\left\lvert \sum_{j=1}^{m}\sum_{k=1}^{n} c_{jk} z_j y_k \right\rvert
\leq \sqrt{RC}\left(\sum_{j=1}^{m} |z_j|^2 \right)^{\frac{1}{2}}
\left(\sum_{k=1}^{n} |y_k|^2 \right)^{\frac{1}{2}}
\]
where \(R\) and \(C\) are respectively the row and column sum maxima defined by 
\[
R = \max_{j} \sum_{k=1}^{n} |c_{jk}| \text{ and } C = \max_{k} \sum_{j=1}^{n} |c_{jk}|.
\]
\end{lemma}

	Let \(u\) and \(v\) be the smallest and largest integers, respectively, such that \(\mathbf{j}(u), \mathbf{j}(v) = 1\). The idea will be to change the order of summation of \(R_{\mathbf{j}}\) so that it resembles the form of the double sum in Lemma \ref{2dpR}. Repeated applications of the triangle inequality will then give upper bounds for the quantities corresponding to \(R\) and \(C\) in Lemma \ref{RC}, allowing for its application. Finally, the resulting upper bound sum will be bounded via Lemma \ref{2dpR}. Specifically, define the sequences
\begin{align*} 
U_{x_{u+1}} &:= \left\{S_{E,\alpha_u}(x_{u+1}): x_{u+1} \in E \right\} \\ 						 	
V_{x_v}&:=   \left\{S_{E,\alpha_v}(x_v): x_v \in E \right\}
\end{align*}
corresponding to the sequences \(\{z_j\}\) and \(\{y_k\}\) in the lemma. The terms of the double sequence
corresponding to \(\{c_{jk}\}\) will be the remaining part of \(R_{\mathbf{j}}\) after the sums \(S_{E,\alpha_u}(x_{u+1})\) and \(S_{E,\alpha_{v+1}}(x_v)\) are factored out. We call this double sequence \(W_{u+1,v}\). Note that this factoring out is possible because \(x_u\) and \(x_{v+1}\) each appear in only one \(\chi\) expression in \(R_{\mathbf{j}}\). \\
     
Now if \(m\) is the number of distinct variables \(x_i\) appearing in the summand of \(R_{\mathbf{j}}\), it follows after repeated applications of the triangle inequality and the fact that complex exponentials have absolute value 1 that
\[
R = \max_{x_{u+1}} \sum_{x_v \in E}|W_{x_{u+1}, x_v}| \leq q^{n-2}|E|^{m - 3}
\]
as all but 3 of the index variables \(x_i \in E\) and all but 2 of the index variables \(s_i \in \mathbb{Z}_q \setminus \{0\}\) vary in \(\sum_{x_v \in E}|W_{x_{u+1}, x_v}|\). The column maximum \(C\) can be bounded in the same fashion. Thus, all that remains before we can apply Lemma \ref{RC} is to bound the exponent of \(|E|\) that appears upon simplifying \(R_{\mathbf{j}}\). We will do this with the following simple counting argument.
\begin{lemma}\label{mn}
Let \(\mathbf{j}\) be a binary k-tuple, and suppose that \(m\) distinct variables \(x_i\) appear in the associated sum \(R_{\mathbf{j}}\). Then the number of integers \(1 \leq i \leq k\) such that \(s_i = s_{i-1} = 0\) is bounded above by \(k - m + 1\).
\end{lemma}
\begin{proof}
Let \( n \) be the number of nonzero entries. Define \(z = |\{i : s_i = s_{i-1} = 0\}|\) and \(z' = |\{i : s_i = 0, s_{i-1} \neq 0\}|\). As exactly \(n\) entries of \(\mathbf{j}\) are nonzero, it is clear that \(z + z' = k - n\). Now for each \(i\) such that \(s_i \neq 0\), consider the expression \(\chi(s_i(x_i \cdot x_{i+1} - \alpha_i))\). If \(s_{i-1} \neq 0\) as well, then only \(x_{i+1}\) appears for the first time (assuming the \(\chi\) expressions are written in order by the indices of their auxiliary variables) in \(R_{\mathbf{j}}\) in this expression. However, if \(s_{i-1} = 0\), then both \(x_i\) and \(x_{i+1}\) appear for the first time in this expression. Thus, \(m = n + a\) where \(a\) is the number of integers \(i\) such that \(s_i \neq 0\) and \(s_{i-1} = 0\).

Next, let \(i\) be an integer such that \(s_i \neq 0\) and \(s_{i-1} = 0\) but \(i\) is not the smallest integer such that \(\mathbf{j}(i) =1 \). If no such \(i\) exists, then the nonzero entries of \(\mathbf{j}\) must be consecutive, and so \(m = n + 1\) and \(z \leq k - n = k - m - 1\). Otherwise, there is a largest integer $b$ with \(1 \leq b < i - 1\) such that \(s_b \neq 0\) and \(s_c = 0\) for every integer \(c\) satisfying \(b < c < i\). Hence, \(s_{b+1}\) contributes to the quantity \(z'\). It follows then that \(z' \geq a - 1\), and therefore
\[
z = k - n - z' \leq k - n - a + 1 = k - m + 1.
\]
\end{proof}

We have by Lemma \ref{mn} that we can estimate the exponent of $|E|$ out front, then continue by applying Lemma \ref{RC}, bound the sums over $E$ by the sums over $\mathbb Z_q^d,$ and apply Lemma \ref{2dpR}.
\[
|R_{\mathbf{j}}| \leq q^{-k}|E|^{k - m + 1} \times
\]
\[\hskip 5em\left\lvert \sum_{x_{u+1} \in E} \sum_{x_v \in E}\chi(s_u(x_u \cdot x_{u+1} - \alpha_u))\chi(s_v(x_v \cdot x_{v+1} - \alpha_v)) W_{x_{u+1}, x_v} \right\rvert
\]
\[
\leq q^{-k}|E|^{k - m + 1} q^{n-2} |E|^{m-3}\left(\sum_{x_{u+1} \in E} \left\lvert S_{E,\alpha_u}(x_{u+1})\right\rvert ^2 \right)^{\frac{1}{2}} \left(\sum_{x_{v} \in E} \left\lvert S_{E,\alpha_v}(x_v)\right\rvert ^2 \right)^{\frac{1}{2}}\]
\[
\leq q^{n-k-2}|E|^{k-2} \left(\sum_{x_{u+1} \in \mathbb{Z}_q^d} \left\lvert S_{E,\alpha_u}(x_{u+1})\right\rvert ^2 \right)^{\frac{1}{2}} \left(\sum_{x_{v} \in \mathbb{Z}_q^d} \left\lvert S_{E,\alpha_v}(x_v)\right\rvert ^2 \right)^{\frac{1}{2}}
\]
\[					    
\leq 2q^{n-k-2}|E|^{k-1}q^{\frac{d(2\ell - 1)}{\ell} + \frac{1}{\ell}}.
\]
	Thus, for $n$ nonzero entries, we get
\begin{equation}\label{Rn}
|R_n| \leq 2\binom{k}{n}q^{n-k-2}|E|^{k-1}q^{\frac{d(2\ell - 1)}{\ell} + \frac{1}{\ell}}.
\end{equation}
This completes the proof since
\[
|\Pi_\alpha (E)| = \frac{|E|^{k+1}}{q^k} + \sum_{i=1}^k R_i,
\]
and we have that \eqref{R1}, \eqref{R2}, and \eqref{Rn} together show that
\[
\left|\sum_{i=1}^k R_i\right| \lesssim q^{-k}|E|^{k}q^{\left(\frac{d-1}{2}\right)(2 - \frac{1}{\ell}) + 1} + \sum_{n=2}^{k} q^{n - k - 2}|E|^{k-1}q^{\frac{d(2\ell - 1)}{\ell} + \frac{1}{\ell}} \lesssim |E|^{k+1}q^{-k}\]
provided \(|E| \gtrsim q^{\frac{k-2}{2} + \frac{d(2\ell - 1) + 1}{2\ell}}\).

\end{proof}

\section{Bounds on \(k\)-chains in \(\mathbb F_q^d\)}

A similar result holds for \(k\)-chains of dot products in vector spaces over finite fields of odd order. 

\begin{theorem}\label{kF}
	Let \(E \subseteq \mathbb{F}_q ^d\) where \(q\) is a power of an odd prime, and let \(\mathbf{\alpha} = (\alpha_1, \alpha_2, \cdots, \alpha_k)\) be a \(k\)-tuple of elements in \(\mathbb{F}_p\) where \(k \geq 2\). Then we have
\[
|\Pi_\alpha (E)| = \frac{|E|^{k+1}}{q^k}(1 + o(1))
\]
provided \(|E| \gtrapprox q^{\frac{d+k}{2}}\) or \(|E| \gtrsim q^{\frac{d+k-1}{2}}\) if the components of \( \mathbf{\alpha}\) are nonzero.
\end{theorem}

\begin{proof}
The proof of this result is mostly analogous to that of the previous theorem. Note that Lemmas \ref{RC} and \ref{mn} still hold in this case. Now fix integers \(d, k \geq 1\) and a power of an odd prime \(q\) and consider a subset \(E\) of \(\mathbb{F}_q ^d\). If \(\mathbf{\alpha} = (\alpha_1, \alpha_2, \cdots, \alpha_k)\) is a \(k\)-tuple of elements in \(\mathbb{F}_q\), then we have by orthogonality
\[
|\Pi_\alpha (E)| = q^{-k} \sum_{\substack{s_j \in \mathbb{F}_q \\ 1 \leq j \leq k}}\sum_{\substack{x_j \in E \\ 1 \leq j \leq k+1}}\prod_{j=1}^{k} \chi(s_j(x_j \cdot x_{j+1} - \alpha_j)).
\]
As in the previous proof, we decompose this sum into different remainder terms based on which \(s_i\) are nonzero. Let \(R_n\) (\(n \geq 0\)) have the meaning analogous to the similarly named quantity from the previous proof. It is easy to see that
\[
|R_0| = q^{-k}|E|^{k+1}.
\]
Next, consider a binary \(k\)-tuple \textbf{j} with exactly one entry equal to 1. The remainder sum \(R_\mathbf{j}\) is of the form of the sum in Lemma \ref{1dp}, and hence
\[
|R_\mathbf{j}| \leq q^{-k}|E|^{k-1}|E| q^{\frac{d+1}{2}}\lambda(\alpha_i) = q^{-k}|E|^k q^{\frac{d+1}{2}}\lambda(\alpha_i)
\]
where \(\mathbf{j}(i) \neq 0\).
Thus, 
\[
|R_1| \leq \sum_{i=1}^{k} q^{-k}|E|^k q^{\frac{d+1}{2}}\lambda(\alpha_i)
\]
	
The upper bound for \(|R_2|\) is obtained in an analogous way to the method of the preceding proof but with the application of Lemma \ref{2dp}. As there are \(\binom{k}{2}\) remainder sums of this type, it follows that	
\[
|R_2| \leq \sum_{1 \leq i \neq j \leq k} q^{d-k+1}|E|^{k-1}\lambda(\alpha_i)\lambda(\alpha_j)
\]
Lastly, consider \(R_\mathbf{j}\) where exactly \(3 \leq n \leq k\) entries of \textbf{j} are nonzero. Suppose that there are \(m\) distinct \(x_i\) appearing in \(R_\mathbf{j}\), and let \(u\) and \(v\) be, respectively, the smallest and largest integers such that \(\mathbf{j}(u), \mathbf{j}(v) \neq 0 \). Following the reasoning in the proof of Theorem \ref{mainR} and applying Lemma \ref{2dp}, we have
\[
|R_\mathbf{j}| \leq q^{d+n-k-1}|E|^{k-1}\lambda(\alpha_u)\lambda(\alpha_v)
\]
It follows then, after dominating the sum of the lambda factors, that
\[
|R_n| \leq \sum_{1 \leq i \neq j \leq k} q^{d+n-k-1}|E|^{k-1}\lambda(\alpha_i)\lambda(\alpha_j).
\]
The result now follows as we have shown that
\[
|\Pi_\alpha (E)| = |E|^{k+1}q^{-k} + \sum_{i=1}^k R_i
\]
where
\[
\left|\sum_{i=1}^k R_i \right|\lesssim \sum_{i=1}^{k} q^{-k}|E|^k q^{\frac{d+1}{2}}\lambda(\alpha_i) +
\sum_{1 \leq i \neq j \leq k} q^{d-k+1}|E|^{k-1}\lambda(\alpha_i)\lambda(\alpha_j) +
\]
\[
\sum_{n=3}^{k}\sum_{1 \leq i \neq j \leq k} q^{d+n-k-1}|E|^{k-1}\lambda(\alpha_i)\lambda(\alpha_j),
\]
and the claimed result holds.
\end{proof}

\section{Small set results}\label{smallSection}

We now turn to the estimates that apply to smaller subsets of $\mathbb F_q^2.$ We begin by stating a result from \cite{CS} (the corrected form of Theorem 3 from that paper). We then derive Theorem \ref{smallSet} as a corollary. Finally, we describe the error in \cite{CS}.

\begin{theorem}[Theorem 3 from \cite{CS}]\label{smallSet2}
Given $E\subseteq \mathbb F_q^2,$ with $q$ a power of an odd prime, and $\alpha,\beta \in \mathbb F_q^*,$,
$$\left|\Pi_{(\alpha,\beta)}(E)\right|\lesssim |E|^2.$$
\end{theorem}

\subsection{Proof of Theorem \ref{smallSet}}

We start by assuming that $(k+1)$ is a multiple of three. Notice that for any $(\alpha_1, \dots, \alpha_k)\in \mathbb F_q^k,$ with all $\alpha_j$ nonzero, every dot product $k$-chain of this type must have a triple of points $(x_1, x_2, x_3)$ forming a dot product 2-chain of type $(\alpha_1,\alpha_2),$ followed by a triple of points $(x_4, x_5, x_6)$ forming a dot product 2-chain of type $(\alpha_4,\alpha_5),$ and so on, for a total of $(k+1)/3$ triples of points. Theorem \ref{smallSet2} tells us that there are no more than $|E|^2$ choices of triples of each type. So when $(k+1)$ is a multiple of three, the total number of dot product $k$-chains can be no more than
$$\lesssim \left(|E|^2\right)^\frac{k+1}{3}=|E|^\frac{2(k+1)}{3}.$$

Now, if $(k+1)$ is one more than a multiple of three, we reason as above to get that there are no more than
$$\lesssim |E|^\frac{2k}{3}$$
occurrences of dot product $(k-1)$-chains of the type $(\alpha_1, \dots, \alpha_{k-1}),$ accounting for all choices of the relevant $k$-tuples $(x_1, \dots, x_k).$ Then there are no more than $|E|$ choices for the point $x_{k+1},$ giving us a total upper bound of no more than
$$\lesssim |E|^{\frac{2k}{3}+1}.$$
Notice that in this case, we are not using the fact that the $x_k \cdot x_{k+1} = \alpha_k$ at all.

Finally, if $(k+1)$ is two more than a multiple of three, we again reason as before to get that there are no more than
$$\lesssim |E|^\frac{2(k-1)}{3}$$
occurrences of dot product $(k-2)$-chains of the type $(\alpha_1, \dots, \alpha_{k-2}),$ and now have $|E|$ choices for $x_k$, as well as $|E|$ choices for $x_{k+1}$, for a total upper bound of no more than
$$\lesssim |E|^{\frac{2(k-1)}{3}+2}.$$

Combining these yields the claimed upper bound on dot product $k$-chains of
$$\lesssim |E|^{\left\lceil \frac{2(k+1)}{3}\right\rceil}.$$

\subsection{Erratum}

We now detail the error in Theorem 3 from \cite{CS}. In that paper, it was claimed that an analog of Theorem \ref{smallSet2} above held for $\mathbb Z_q^d.$ However, there is a mistake in the proof. First we present a construction showing that the claim is false. Then we describe how the method employed fails.

\begin{proposition}
Given $q=p^\ell,$ for a natural number $\ell\geq 2$ and an odd prime $p,$ and units $\alpha, \beta \in \mathbb Z_q,$ there exists a subset of $E \subseteq \mathbb Z_q^2$ with $$\left|\Pi_{(\alpha,\beta)}(E)\right|= |E|^3.$$
\end{proposition}
\begin{proof}
Given distinct units $\alpha$ and $\beta,$ both different from 1, we define $E\subseteq Z_q^2$, to be the union of three sets:
$E=X\cup Y \cup Z$ where,
$$X=\{(ap,\alpha): a\in Z_p\},\ Y=\{(bp^{\ell-1},1): b\in Z_p\},\ Z=\{(cp,\beta) : c\in Z_p\}.$$
By definition, we have $|E|= |X|+|Y|+|Z|=3p.$ Now notice that for any $x\in X,$ and any $y\in Y,$ we have that for some $a,b\in\mathbb Z_p,$
$$x\cdot y = (ap)\left(bp^{\ell -1}\right)+(\alpha)(1) = abp^\ell+\alpha=\alpha.$$
Similarly, for any $y\in Y,$ and any $z\in Z,$ we will have $y\cdot z = \beta.$ Therefore we can compute explicitly,
$$|\Pi_{\alpha, \beta}(E)|\geq |\{(x,y,z) \in X\times Y\times Z : x\cdot y = \alpha, x\cdot z = \beta\} = p^3=|E|^3.$$
\end{proof}

The issue lies in how incidences were counted in $\mathbb Z_q^d.$ Given a unit $\alpha\in\mathbb Z_q,$ where $q=p^\ell,$ for a natural number $\ell\geq 2$ and an odd prime $p,$ and an element $v\in\mathbb Z_q^2,$ we define
$$L\alpha(x):=\{y\in \mathbb Z_q^2: x\cdot y = \alpha\}.$$
Looking through the proof of Theorem 3 in \cite{CS}, we find the incorrect claim that if $w$ and $v$ are distinct elements in $\mathbb Z_q^2,$ with $|L_\alpha(v)\cap L_\beta(w)|>1$, then $L_\alpha(v) = L_\beta(w)$. While the analogous result holds for $v,w\in\mathbb F_q^2,$ we give an explicit counterexample for this statement in $Z_9^2$.
Let $v=(3,2)$ and $w=(3,4).$ Notice that
$$L_2(v)=\{(0,1), (1,4), (2,7), (3,1), (4,4), (5,7), (6,1), (7,4), (8,7)\},\text{ and}$$
$$L_4(w)=\{(0,1), (1,7), (2,4), (3,1), (4,7), (5,4), (6,1), (7,7), (8,4)\},$$
yet
$\left|L_2(v)\cap L_4(w)\right|=\left|\{(0,1),(3,1),(6,1)\}\right|=3>1,$ but we can see that $L_2(v) \neq L_4(w),$ contradicting the erroneous claim.







\end{document}